\documentclass[11pt]{amsart}
\usepackage{amsmath}
\usepackage{amssymb}
\usepackage{amsthm}
\usepackage{amscd}
\usepackage{enumerate}
\usepackage{enumitem}
\usepackage{verbatim}
\usepackage[all]{xy}
\usepackage{mathrsfs}
\usepackage{mathabx}
\usepackage{url}
\usepackage{hyperref}

\theoremstyle{plain}
\newtheorem{thm}{Theorem}[section]
\newtheorem{prop}[thm]{Proposition}
\newtheorem{lem}[thm]{Lemma}
\newtheorem{cor}[thm]{Corollary}

\theoremstyle{definition}

\newtheorem{rmk}[thm]{Remark}

\newcommand{\Aut}{\mathrm{Aut}}

\DeclareMathOperator{\Coker}{Coker}

\newcommand{\EV}{\mathrm{EV}}

\newcommand{\har}{\mathrm{har}}

\newcommand{\Hom}{\mathrm{Hom}}

\DeclareMathOperator{\Ind}{Ind}

\DeclareMathOperator{\Ker}{Ker}

\newcommand{\opp}{\mathrm{op}}

\DeclareMathOperator{\Res}{Res}

\newcommand{\st}{\mathrm{st}}
\newcommand{\St}{\mathrm{St}}

\newcommand{\cL}{\mathcal{L}}

\newcommand{\cO}{\mathcal{O}}

\newcommand{\cT}{\mathcal{T}}

\newcommand{\cV}{\mathcal{V}}

\newcommand{\bC}{\mathbb{C}}

\newcommand{\bF}{\mathbb{F}}

\newcommand{\bP}{\mathbb{P}}
\newcommand{\bQ}{\mathbb{Q}}

\newcommand{\bZ}{\mathbb{Z}}

\newcommand{\frem}{\mathfrak{m}}
\newcommand{\frn}{\mathfrak{n}}

\newcommand{\Kinf}{K_\infty}
\newcommand{\Cinf}{\mathbb{C}_{\infty}}

\makeatletter
    
    \@addtoreset{equation}{section}
\makeatother

\makeatletter
\renewcommand{\p@enumii}{}
\makeatother

\begin{document}

\title[Constancy of Hecke eigensystems for Drinfeld cuspforms]{On some constancy of Hecke eigensystems for Drinfeld cuspforms of finite slope}

\author{Shin Hattori}
\address{Department of Natural Sciences, Tokyo City University, 1-28-1 Tamazutsumi, Setagaya-ku, Tokyo 158-8557, Japan}

\date{\today}


\begin{abstract}
	Let $p$ be a rational prime, let $q>1$ be a $p$-power integer, let $\bF_q$ be the field of $q$ elements and let $A=\bF_q[t]$ be the polynomial ring over $
	\bF_q$. Let $\frn\in A$ be a nonzero element and let $\wp\in A$ be a monic irreducible polynomial of positive degree.
	Let $k\geq 2$ and $r\geq 1$ be integers.
	Let $S_k(\Gamma_1(\frn\wp^r))$ be the space of Drinfeld cuspforms of level $\Gamma_1(\frn\wp^r)$ and weight $k$. 
	In this paper, we prove that the multiplicity of a Hecke eigensystem of finite $\wp$-slope in $S_k(\Gamma_1(\frn\wp^r))$ is equal to $q^{(r-1)\deg(\wp)}$ times 
	that in $S_k(\Gamma_1(\frn\wp))$.
	In particular, this shows that a Hecke eigensystem of finite $\wp$-slope appears in $S_k(\Gamma_1(\frn\wp^r))$ if and only if it appears in 
	$S_k(\Gamma_1(\frn\wp))$.
\end{abstract}

\maketitle



\section{Introduction}

Let $p$ be a rational prime, let $q>1$ be a $p$-power integer, let $\bF_q$ be the field of $q$ elements, let $t$ be an indeterminate, let $A=\bF_q[t]$, $K=\bF_q(t)
$, $K_\infty=\bF_q((1/t))$ and $\cO_{\Kinf}=\bF_q[[1/t]]$. Let $\Cinf$ be the $(1/t)$-adic completion of an algebraic closure of $\Kinf$. Let $
\Omega=\Cinf\setminus \Kinf$ be the Drinfeld 
upper half plane, which is equipped with a natural structure of a rigid analytic variety over $\Cinf$. 

Let $\Pi$ be the set of monic irreducible polynomials of positive degree in $A$. Let $\frn\in A$ be a nonzero element 
and let $\wp\in \Pi$ satisfy $\wp\nmid\frn$. Let $k\geq 2$ and $r\geq 1$ be integers. 
For any nonzero $\frem\in A$, let
\[
\Gamma_1(\frem)=\left\{\gamma\in \mathit{SL}_2(A)\ \middle|\ \gamma\equiv\begin{pmatrix}
	1 & *\\0& 1
\end{pmatrix}\bmod \frem\right\}.
\]

A Drinfeld modular form is a rigid analytic function $f:\Omega\to \Cinf$ satisfying a transformation condition and regularity at cusps similar to those for elliptic modular forms. Though the definitions of elliptic and Drinfeld modular forms are parallel, arithmetic of Drinfeld modular forms is sometimes very different from the elliptic case. 

It has been realized that Drinfeld modular forms of level $\Gamma_1(\frem)$ behave rather strangely. For example, for a positive integer $M$, the space of elliptic 
modular forms of level $\Gamma_1(M)$ is written as the direct sum 
of the spaces of level $\Gamma_0(M)$ with nebentypus. On the other hand, such a 
decomposition is not possible in general for the space of Drinfeld modular forms of level $\Gamma_1(\frem)$, since representations of  $(A/\frem A)^\times$ over $
\Cinf$ are not necessarily semisimple.
    
In \cite{Ha_DMO}, we proved an unexpected triviality of the Hecke action on the $t$-ordinary part of the space of Drinfeld cuspforms of level $\Gamma_1(t^r)$. The 
aim of the present paper is to generalize it to the finite slope part for more general levels: it turns out that not the triviality but a constancy in some sense 
holds 
for Hecke eigensystems appearing in the finite $\wp$-slope part of the space of Drinfeld cuspforms of level $\Gamma_1(\frn\wp^r)$.

To state the main theorem, we fix some notation. Let $S_k(\Gamma_1(\frn\wp^r))$ be the $\Cinf$-vector space of Drinfeld cuspforms of level
$\Gamma_1(\frn\wp^r)$ and weight $k$, on which the Hecke operator $T_Q$ at $Q$ acts for any $Q\in \Pi$. We write $T_Q$ also as $U_Q$ when $Q\mid \frn\wp^r$.
For any $\alpha\in \Cinf$, we denote by $S_k(\Gamma_1(\frn\wp^r))(T_Q-\alpha)$ the generalized eigenspace of $T_Q$ belonging to $\alpha$.

Let $\bar{K}$ be the algebraic closure of $K$ in $\Cinf$. Let $K_\wp$ be the completion of $K$ at $\wp$ and let $\bar{K}_\wp$ be an algebraic
closure of $K_\wp$. Let $v_\wp$ be the $\wp$-adic additive valuation on $\bar{K}_\wp$ satisfying $v_\wp(\wp)=1$.
We choose once and for all an embedding $\iota_\wp: \bar{K}\to \bar{K}_\wp$ of $K$-algebras. 
For any $a\in \bQ\cup\{+\infty\}$, we say that an element $\alpha\in \bar{K}$ is of slope $a$ if
$v_\wp(\iota_\wp(\alpha))=a$. 
For any $\alpha\in \bar{K}$, we denote by $m(\Gamma_1(\frn\wp^r),k,\wp,\alpha)$ the multiplicity of $\alpha$ as an eigenvalue of $U_\wp$ acting on
$S_k(\Gamma_1(\frn\wp^r))$. For any $a\in \bQ\cup\{+\infty\}$,
we denote by $d(\Gamma_1(\frn\wp^r),k,a)$ the multiplicity of $U_\wp$-eigenvalues of slope $a$ appearing in $S_k(\Gamma_1(\frn\wp^r))$. 

For any $
\lambda=(\lambda_Q)_{Q\in \Pi}\in \bar{K}^\Pi$, we say that $\lambda$ is a Hecke eigensystem appearing in $S_k(\Gamma_1(\frn\wp^r))$
if there exists a nonzero element $f\in S_k(\Gamma_1(\frn\wp^r))$ satisfying $f|T_Q=\lambda_Q f$ for all $Q\in \Pi$, and denote by 
\[
m(\Gamma_1(\frn\wp^r),k,\lambda)=\dim_{\Cinf}\bigcap_{Q\in \Pi} S_k(\Gamma_1(\frn\wp^r))(T_Q-\lambda_Q)
\]
the multiplicity of the Hecke eigensystem $\lambda$ appearing in $S_k(\Gamma_1(\frn\wp^r))$.

Now our main theorem is as follows.

\begin{thm}\label{ThmMainIntro}
	\begin{enumerate}
		\item\label{ThmMainIntro_Slope}(Corollary \ref{CorMultiSlope}) For any $\alpha\in \bar{K}^\times$ and $a\in \bQ$, we have
		\[
		\begin{aligned}
			m(\Gamma_1(\frn\wp^r),k,\wp,\alpha)&=q^{(r-1)\deg(\wp)}m(\Gamma_1(\frn\wp),k,\wp,\alpha),\\
			d(\Gamma_1(\frn\wp^r),k,a)&=q^{(r-1)\deg(\wp)}d(\Gamma_1(\frn\wp),k,a).
		\end{aligned}
		\]
		In particular, the sets of nonzero $U_\wp$-eigenvalues ({\textit resp.} finite slopes) appearing in $S_k(\Gamma_1(\frn\wp^r))$ and $S_k(\Gamma_1(\frn\wp))$ 
		are the same.
		\item\label{ThmMainIntro_Eigensys}(Theorem \ref{ThmConstancy}) Let $\lambda=(\lambda_Q)_{Q\in\Pi}\in \bar{K}^\Pi$ satisfy $\lambda_\wp\neq 0$. Then 
		we have
		\[
		m(\Gamma_1(\frn\wp^r),k,\lambda)=q^{(r-1)\deg(\wp)}m(\Gamma_1(\frn\wp),k,\lambda).
		\]
		In particular, $\lambda$ is a Hecke eigensystem appearing in $S_k(\Gamma_1(\frn\wp^r))$ if and only if it is a Hecke eigensystem appearing in 
		$S_k(\Gamma_1(\frn\wp))$.
	\end{enumerate}
\end{thm}

This shows a stark contrast to the case of elliptic modular forms where Hecke eigensystems of finite $p$-slope vary in $p$-adic analytic families 
when the weight and the $p$-adic valuation of the level vary. Though in some cases we have $\wp$-adic continuous families of Drinfeld eigenforms
varying the weight \cite{Ha_DMBF}, Theorem \ref{ThmMainIntro} indicates that varying the level only yields constant families of Hecke eigensystems
of finite slope appearing in the space of Drinfeld modular forms.

Let us give an idea of the proof of Theorem \ref{ThmMainIntro}. Let $\Theta_r$ be the multiplicative group $1+\wp(A/\wp^r A)$. It is a commutative $p$-group acting 
on $S_k(\Gamma_1(\frn\wp^r))$ 
via diamond operators. Let $\Cinf[\Theta_r]$ be the group ring of $\Theta_r$ over $\Cinf$.
Then the key point of the proof is the freeness of the $\Cinf[\Theta_r]$-module $S_k(\Gamma_1(\frn\wp^r))$ (Proposition \ref{PropFree}). 
Since $\Cinf$ is of characteristic $p$, the ring $\Cinf[\Theta_r]$ is local and any direct summand of the free $
\Cinf[\Theta_r]$-module $S_k(\Gamma_1(\frn\wp^r))$ is also free. This applies in particular to generalized eigenspaces of each Hecke operator. Since the $
\Theta_r$-fixed part of $\Cinf[\Theta_r]$ is a $\Cinf$-vector space of dimension one, this relates the multiplicity of a 
$U_\wp$-eigenvalue or a Hecke eigensystem in $S_k(\Gamma_1(\frn\wp^r))$ with that in $S_k(\Gamma_1(\frn\wp^r))^{\Theta_r}=S_k(\Gamma_1^p(\frn,
\wp^r))$, where
\[
\Gamma_1^p(\frn,\wp^r)=\left\{\gamma\in \Gamma_1(\frn)\ \middle|\ \gamma\bmod \wp^r\in\begin{pmatrix}
	\Theta_r & *\\0& \Theta_r
\end{pmatrix}\right\}.
\]
Then the theorem follows from the fact that the inclusion $S_k(\Gamma_1(\frn\wp))\to S_k(\Gamma_1^p(\frn,\wp^r))$ induces a Hecke equivariant isomorphism on finite slope parts (Lemma \ref{LemFiniteSlopeGenES}). 

The organization of this paper is as follows. In \S\ref{SecHC}, we recall descriptions of Drinfeld cuspforms using harmonic cocycles and the Steinberg module 
\cite{Tei}. In \S\ref{SecHecke}, we recall the definition of Hecke operators and diamond operators, and
give the aforementioned isomorphism between the finite slope part of level $\Gamma_1(\frn\wp)$ and that of level $\Gamma_1^p(\frn,\wp^r)$.
In \S\ref{SecFree}, we prove the key freeness of the $\Cinf[\Theta_r]$-module 
$S_k(\Gamma_1(\frn\wp^r))$. In \S\ref{SecConst}, we show Theorem \ref{ThmMainIntro}. 

\subsection*{Acknowledgments} This work was supported by JSPS KAKENHI Grant Number 	26K06733.

\section{Combinatorial descriptions of Drinfeld cuspforms}\label{SecHC}

In this section, we recall the descriptions of the space of Drinfeld cuspforms using harmonic cocycles and the Steinberg module due to Teitelbaum \cite[p.~506]
{Tei},
following the normalization of \cite[\S5.3]{Boeckle} (see also \cite[\S2]{Ha_DMBF}).

For any $A$-algebra $B$, we consider $B^2$ as the set of row vectors which admits a left action of $\mathit{GL}_2(B)$ by $\gamma\circ (x,y)=(x,y)\gamma^{-1}$.
Let $\cT$ be the Bruhat--Tits tree for $\mathit{SL}_2(\Kinf)$. We denote by $\cT_0$ the set of vertices of $\cT$ and by $\cT_1^o$ the set of oriented edges of $
\cT$. 
By 
definition, the set $\cT_0$ consists of 
$\Kinf^\times$-equivalence classes of $\cO_{\Kinf}$-lattices in $\Kinf^2$. For any $e\in \cT_1^o$, the origin, the terminus and the opposite edge of $e$ are denoted 
by $o(e)$, $t(e)$ and $-e$, respectively. Then the group $\{\pm 1\}$ acts on $\cT_1^o$ by $(-1)e=-e$. 

Let $\Gamma$ be a congruence subgroup of $\mathit{SL}_2(A)$ which is $p'$-torsion free (that is, every element of $\Gamma$ of finite order has $p$-power order).
Consider its action $\circ$ on $\cT$ via the natural inclusion $\Gamma\to \mathit{GL}_2(\Kinf)$. We say that 
a vertex or an oriented edge of $\cT$ is $\Gamma$-stable if 
its 
stabilizer in $
\Gamma$ is trivial. We denote by $\cT_0^{\Gamma\text{-}\st}$ and $\cT_1^{o,\Gamma\text{-}\st}$ the sets of $\Gamma$-stable vertices 
and oriented edges, respectively. They are stable under the action of $\Gamma$. More generally, let $\tilde{\Gamma}$ be any subgroup
of $\mathit{SL}_2(A)$ satisfying $\tilde{\Gamma}\rhd \Gamma$.
Then for any $s\in \cT_0\sqcup \cT_1^o$ and $\tilde{\gamma}\in \tilde{\Gamma}$, the element $s$ is $\Gamma$-stable if and only if $\tilde{\gamma}\circ s$ is $
\Gamma$-stable. Thus the sets $\cT_0^{\Gamma\text{-}\st}$ and $\cT_1^{o,\Gamma\text{-}\st}$ are stable under the action of 
$\tilde{\Gamma}$.

For any set $S$, we denote by $\bZ[S]$ the free $\bZ$-module with basis $\{[s]\mid s\in S\}$. 
Let
\[
\bZ[\bar{\cT}_1^{o,\Gamma\text{-}\st}]:=\bZ[\cT_1^{o,\Gamma\text{-}\st}]/\langle [e]+[-e]\mid e\in \cT_1^{o,\Gamma\text{-}\st}\rangle.
\]
For any $v\in \cT_0$, define an element $[v]^{\Gamma\text{-}\st}\in \bZ[\cT_0^{\Gamma\text{-}\st}]$ by $[v]^{\Gamma\text{-}\st}=[v]$ if $v$ is $\Gamma$-stable
 and $[v]^{\Gamma\text{-}\st}=0$ otherwise. Let
\[
\partial_\Gamma:\bZ[\bar{\cT}_1^{o,\Gamma\text{-}\st}]\to \bZ[{\cT}_0^{\Gamma\text{-}\st}],\quad [e]\mapsto [t(e)]^{\Gamma\text{-}\st}-[o(e)]^{\Gamma\text{-}
\st}.
\]
For $\tilde{\Gamma}\supseteq \Gamma$ as above, the map $\partial_\Gamma$ is $\tilde{\Gamma}$-equivariant.

We refer to the kernel of the augmentation map
\[
\St:=\Ker(\bZ[\bP^1(K)]\to \bZ)
\]
as the Steinberg module. We consider $\St$ as a left $\bZ[\Gamma]$-module via
\[
\gamma\circ (x:y)=(x:y)\gamma^{-1},\quad (x:y)\in \bP^1(K).
\]
Then the left $\bZ[\Gamma]$-module $\St$ is finitely generated and projective. 
Moreover, there exists a split exact sequence of left $\bZ[\Gamma]$-modules
\begin{equation}\label{EqnExactSt}
	\xymatrix{
		0 \ar[r]& \St \ar[r]& \bZ[\bar{\cT}_1^{o,\Gamma\text{-}\st}] \ar[r]^{\partial_\Gamma} & \bZ[\cT_0^{\Gamma\text{-}\st}]\ar[r]& 0
	}
\end{equation}
(\cite[\S5,3]{Boeckle}, \cite[Ch.~II, \S2.9]{Se}).

Let $k\geq 2$ be an integer and let $H_{k-2}(\Cinf)$ be the $\Cinf$-subspace of the polynomial ring $\Cinf[X,Y]$ consisting of homogeneous polynomials of degree 
$k-2$. It is equipped with a left action of $\mathit{GL}_2(\Cinf)$ defined by $\gamma\circ (X,Y)=(X,Y)\gamma$. Let
\[
V_k(\Cinf)=\Hom_{\Cinf}(H_{k-2}(\Cinf),\Cinf).
\]
It is also equipped with a natural left action $\circ$ of $\mathit{GL}_2(\Cinf)$. For any $\gamma=\begin{pmatrix}
	a&b\\c&d
\end{pmatrix}\in \mathit{GL}_2(\Cinf)$, $\omega\in V_k(\Cinf)$ and $P(X,Y)\in 
H_{k-2}(\Cinf)$, the action $\circ$ is defined by
\[
(\gamma\circ \omega)(P(X,Y))=\omega(\gamma^{-1}\circ P(X,Y))=\det(\gamma)^{2-k}\omega(P(dX-cY,-bX+aY)).
\]

A harmonic cocycle of level $\Gamma$ and weight $k$ is a map $c:\cT_1^o\to V_k(\Cinf)$ satisfying the following conditions:
\begin{enumerate}
	\item\label{DfnHCEquiv} $c(\gamma\circ e)=\gamma\circ c(e)$ for any $\gamma\in \Gamma$ and $e\in \cT_1^o$,
	\item\label{DfnHCPm} $c(-e)=-c(e)$ for any $e\in \cT_1^o$,
	\item\label{DfnHcHarm} $\sum_{e\in \cT_1^o,\ v=t(e)} c(e)=0$ for any $v\in \cT_0$.
\end{enumerate}
The $\Cinf$-vector space of harmonic cocycles of level $\Gamma$ and weight $k$ is denoted by $C_k^\har(\Gamma)$. 
Similarly, the $\Cinf$-vector space
of maps $c:\cT_1^o\to V_k(\Cinf)$ satisfying the conditions (\ref{DfnHCEquiv}) and (\ref{DfnHCPm}) is denoted by $C_k^\pm(\Gamma)$.

Let $\cV_k(\Gamma)=\St\otimes_{\bZ[\Gamma]}V_k(\Cinf)$ and 
\[
\cL_{1,k}(\Gamma)= \bZ[\bar{\cT}_1^{o,\Gamma\text{-}\st}]\otimes_{\bZ[\Gamma]}V_k(\Cinf),\quad \cL_{0,k}(\Gamma)=\bZ[\cT_0^{\Gamma\text{-}\st}]
\otimes_{\bZ[\Gamma]}V_k(\Cinf),
\]
where the modules on the left of $\otimes$ are considered as right $\bZ[\Gamma]$-modules via the action $s|_\gamma=\gamma^{-1}\circ s$. 
Then (\ref{EqnExactSt}) yields an exact sequence of $\Cinf$-vector spaces
\[
\xymatrix{
	0 \ar[r]& \cV_k(\Gamma) \ar[r]& \cL_{1,k}(\Gamma) \ar[r]& \cL_{0,k}(\Gamma)\ar[r]& 0.
}
\]

Let $\Lambda_1\subseteq \cT_1^{o,\Gamma\text{-}\st}$ be a complete set of representatives of $\Gamma\backslash \cT_1^{o,\Gamma\text{-}\st}/\{\pm 1\}$.
Then the $\Cinf$-linear map
\[
\Phi_\Gamma: C_k^\pm(\Gamma)\to \cL_{1,k}(\Gamma),\quad c\mapsto \sum_{e\in \Lambda_1}[e]\otimes c(e)
\]
is an isomorphism independent of the choice of $\Lambda_1$ \cite[p.~9]{Ha_DMBF}. Moreover,
from \cite[p.~506]{Tei}, we see that it induces a $\Cinf$-linear isomorphism
\begin{equation*}
\Phi_\Gamma: C_k^\har(\Gamma)\to \cV_k(\Gamma).
\end{equation*}

We denote by $S_k(\Gamma)$ the $\Cinf$-vector space of Drinfeld cuspforms of level $\Gamma$ and weight $k$. 
For any rigid analytic function $f:\Omega\to \Cinf$
and $\gamma=\begin{pmatrix}
	a&b\\c&d
\end{pmatrix}\in \mathit{GL}_2(\Kinf)$, let $f|_k\gamma$ be the rigid analytic function on $\Omega$ defined by
\[
(f|_k\gamma)(z)=\det(\gamma)^{k-1}(cz+d)^{-k}f\left(\frac{az+b}{cz+d}\right).
\]
Then we can associate with any $e\in \cT_1^o$ an element $\Res_{\Gamma}(f)(e)\in V_k(\Cinf)$ satisfying
\begin{equation}\label{EqnResSlash}
	\Res_\Gamma(f|_k\gamma)(e)=\gamma^{-1}\circ \Res_\Gamma(f)(\gamma\circ e)
\end{equation}
for any $\gamma\in\mathit{GL}_2(\Kinf)$, which gives a $\Cinf$-linear isomorphism
\begin{equation*}
	\Res_\Gamma: S_k(\Gamma)\to C_k^\har(\Gamma),\quad f\mapsto (e\mapsto \Res_\Gamma(f)(e))
\end{equation*}
(\cite[Theorem 16]{Tei}, \cite[Theorem 5.10]{Boeckle}).

Let $\tilde{\Gamma}$ be a subgroup of $\mathit{SL}_2(A)$ satisfying $\tilde{\Gamma}\rhd \Gamma$. Let $G=\tilde{\Gamma}/\Gamma$.
For any right $\bZ[\tilde{\Gamma}]$-module $M$ and left $\bZ[\tilde{\Gamma}]$-module $N$, the group $\tilde{\Gamma}$ acts from the right on $M\otimes_{\bZ[\Gamma]}
N$ by
\[
(m\otimes n){\tilde{\gamma}}:=m{\tilde{\gamma}}\otimes \tilde{\gamma}^{-1}n
\]
for any $m\in M$, $n\in N$ and $\tilde{\gamma}\in \tilde{\Gamma}$. This induces a right $G$-action on $M\otimes_{\bZ[\Gamma]}
N$ which is functorial on $M$ and $N$. In particular, we have natural right $G$-actions on $\cV_k(\Gamma)$, $\cL_{1,k}(\Gamma)$ and $\cL_{0,k}(\Gamma)$.
On the other hand, the group $\tilde{\Gamma}$ acts on $C_k^\pm(\Gamma)$, $C_k^\har(\Gamma)$ and $S_k(\Gamma)$ from the right by
\[
c\mapsto c|_k\tilde{\gamma}:=(e\mapsto \tilde{\gamma}^{-1}\circ c(\tilde{\gamma}\circ e)),\quad f\mapsto f|_k\tilde{\gamma},
\]
which induce right $G$-actions on them.

\begin{lem}\label{LemQuotAction}
	The isomorphisms 
	\[
	\Phi_{\Gamma}: C_k^{\pm}(\Gamma)\to \cL_{1,k}(\Gamma),\quad \Phi_{\Gamma}\circ \Res_{\Gamma}: S_k(\Gamma)\to \cV_k(\Gamma)
	\]
	are $G$-equivariant.
\end{lem}
\begin{proof}
	Take any $c\in  C_k^{\pm}(\Gamma)$ and $\tilde{\gamma}\in \tilde{\Gamma}$. Let $\Lambda_1$ be a
	complete set of representatives of $\Gamma\backslash \cT_1^{o,\Gamma\text{-}\st}/\{\pm 1\}$. Since $\tilde{\gamma}^{-1}\Gamma\tilde{\gamma}=\Gamma$, the set
	$\tilde{\gamma}\circ\Lambda_1=\{\tilde{\gamma}\circ e\mid e\in \Lambda_1\}$ is also a complete set of representatives of the same coset space. Since the map $
	\Phi_{\Gamma}$ is 
	independent of the choice of $\Lambda_1$, we have
	\[
	\begin{aligned}
		\Phi_{\Gamma}(c|_k\tilde{\gamma})&=\sum_{e\in \Lambda_1} [e]\otimes \tilde{\gamma}^{-1}\circ c(\tilde{\gamma}\circ e)=\sum_{e\in \tilde{\gamma}\circ 
		\Lambda_1}
		[\tilde{\gamma}^{-1}\circ e]\otimes 
		\tilde{\gamma}^{-1}\circ c(e)\\
		&=\sum_{e\in \tilde{\gamma}\circ \Lambda_1}[e]|_{\tilde{\gamma}}\otimes 
		\tilde{\gamma}^{-1}\circ c(e)=\Phi_{\Gamma}(c){\tilde{\gamma}}.
	\end{aligned}
	\]
	Hence the first assertion follows.
	
	By \cite[Remark 1.4]{BGP}, the injection $\St\to \bZ[\bar{\cT}_1^{o,\Gamma\text{-}\st}]$ of (\ref{EqnExactSt}) 
	is $
	\tilde{\Gamma}$-equivariant. Thus the $\Cinf$-subspace $\cV_k(\Gamma)\subseteq \cL_{1,k}(\Gamma)$ is stable under the action of $G$.
	Since $C_k^\har(\Gamma)\subseteq C_k^\pm(\Gamma)$ is also stable under the $G$-action,
	the latter assertion follows from the former and (\ref{EqnResSlash}). 
\end{proof}



\section{Hecke operators and diamond operators}\label{SecHecke}

Let $\frn\in A$ be a nonzero element and let $\wp\in \Pi$ satisfy $\wp\nmid \frn$. For any integer $n\geq 
1$, let $A_n=A/\wp^n A$. Let $r\geq 1$ be an integer and let 
\[
\Theta_r=1+\wp A_r,
\]
which is a subgroup of the multiplicative group $A_r^\times$. 

Let $\Theta$ be any subgroup of $\Theta_r$. Let
\[
\Gamma_1^\Theta(\frn,\wp^r)=\left\{\gamma\in \Gamma_1(\frn)\ \middle|\ \gamma\bmod \wp^r\in \begin{pmatrix}
	\Theta & *\\0& \Theta
\end{pmatrix}\right\}.
\]
Since $\Theta$ is a $p$-group, we see that $\Gamma_1^\Theta(\frn,\wp^r)$ is a congruence subgroup of $\mathit{SL}_2(A)$ which is $p'$-torsion free.
When $\Theta=\Theta_r$, we denote $\Gamma_1^\Theta(\frn,
\wp^r)$ by $\Gamma_1^p(\frn,\wp^r)$. 
Note that $\Gamma_1(\frn\wp^r)=\Gamma_1^{\{1\}}(\frn,\wp^r)$ and we have an isomorphism
\begin{equation}\label{EqnIsomQuotTheta}
\Gamma_1^\Theta(\frn,\wp^r)/\Gamma_1(\frn\wp^r)\to \Theta,\quad \begin{pmatrix}
	a&b\\c&d
\end{pmatrix}\mapsto d\bmod \wp^r.
\end{equation}

\begin{lem}\label{LemGrpRingLocal}
	The group ring $\Cinf[\Theta]$ is an Artinian local ring. In particular, for any free $\Cinf[\Theta]$-module of finite rank, its direct summand
	as a $\Cinf[\Theta]$-module is free of finite rank.
\end{lem}
\begin{proof}
	Since $\Theta$ is a commutative $p$-group, it is isomorphic to $\bigoplus_{i=1}^s \bZ/p^{n_i} \bZ$
	with some nonnegative integers $n_1,\ldots, n_s$. 
	Since the characteristic of $\Cinf$ is $p$, we have an isomorphism of $\Cinf$-algebras
	\[
	\begin{aligned}
	\Cinf[\Theta]&\simeq \Cinf[X_1,\ldots, X_s]/(X_1^{p^{n_1}}-1,\ldots, X_s^{p^{n_s}}-1)\\
	&=\Cinf[X_1,\ldots, X_s]/((X_1-1)^{p^{n_1}},\ldots, (X_s-1)^{p^{n_s}}).
	\end{aligned}
	\]
	This shows the former assertion. For the latter, such a direct summand is finitely generated and projective over the local ring $\Cinf[\Theta]$.
	Thus it is free of finite rank.
\end{proof}

For any $d\in (A/\frn\wp^rA)^\times$, take any matrix $\eta_d\in\mathit{SL}_2(A)$ satisfying
\[
\eta_d\equiv \begin{pmatrix}
	* & *\\0& d
\end{pmatrix}\bmod \frn\wp^r.
\]
As in \cite[\S2.2]{Ha_DMO}, we define the diamond operator $\langle d\rangle_{\frn\wp^r}$ acting on $S_k(\Gamma_1(\frn\wp^r))$ by
\[
f|\langle d\rangle_{\frn\wp^r}=f|_k\eta_d,
\]
which is independent of the choice of $\eta_d$. Since 
\begin{equation*}
	\eta_d \Gamma_1^\Theta(\frn,\wp^r)\eta_d^{-1}= \Gamma_1^\Theta(\frn,\wp^r), 
\end{equation*}
the subspace
$S_k(\Gamma_1^\Theta(\frn,\wp^r))$ is stable under the diamond operators.

Consider the natural isomorphism
\[
\rho:(A/\frn\wp^rA)^\times\to (A/\frn A)^\times \times A_r^\times.
\]
For any $d\in A_r^\times$, let $[1,d]=\rho^{-1}((1,d))$.
Since the map
\[
\Theta\to (A/\frn\wp^rA)^\times,\quad d\mapsto [1,d]
\]
is a group homomorphism, we have a right $\Theta$-action on $S_k(\Gamma_1(\frn\wp^r))$ defined by
\[
\Theta\to \Aut(S_k(\Gamma_1(\frn\wp^r)))^\opp,\quad d\mapsto \langle [1,d]\rangle_{\frn\wp^r}.
\]
By (\ref{EqnIsomQuotTheta}), we have
\begin{equation}\label{EqnThetaFixed}
	S_k(\Gamma_1(\frn\wp^r))^\Theta=S_k(\Gamma_1^\Theta(\frn,\wp^r)).
\end{equation} 

Let $Q\in \Pi$. Write
\[
\Gamma_1^\Theta(\frn,\wp^r)\begin{pmatrix}
	1&0\\0&Q
\end{pmatrix}\Gamma_1^\Theta(\frn,\wp^r)=\coprod_{i\in I(Q)}\Gamma_1^\Theta(\frn,\wp^r) \xi_i.
\]
The Hecke operator $T_Q$ at $Q$ acting on $S_k(\Gamma_1^\Theta(\frn,\wp^r))$
is defined by
\[
f|T_Q =\sum_{i\in I(Q)} f|_k\xi_i.
\]
When $Q\mid \frn\wp^r$, we also write $U_Q$ for $T_Q$.
Since the explicit description of $T_Q$ given in \cite[\S3.1]{Ha_DMBF} is independent of $\Theta$, the natural inclusion
\[
S_k(\Gamma_1^\Theta(\frn,\wp^r))\to S_k(\Gamma_1(\frn\wp^r))
\]
is compatible with $T_Q$ for any $Q$. Moreover, we can show that Hecke operators commute with each other.
By \cite[Lemma 2.3]{Ha_DMO}, the diamond operators commute with all 
Hecke operators.

For a finite dimensional $\Cinf$-vector space $V$, a $\Cinf$-linear map $T:V\to V$ and $\alpha\in \Cinf$, we denote by 
$V(T-\alpha)$ the generalized eigenspace of $T$ belonging to $\alpha$.
It equals $\Ker((T-\alpha)^m)$ for any sufficiently large integer $m$.
Let $\EV(\Gamma_1^\Theta(\frn,\wp^r),k,Q)$ be the set of eigenvalues of $T_Q$
acting on $S_k(\Gamma_1^\Theta(\frn,\wp^r))$. Then we have
\begin{equation}\label{EqnGenESDecomp}
S_k(\Gamma_1^\Theta(\frn,\wp^r))=\bigoplus_{\alpha\in \EV(\Gamma_1^\Theta(\frn,\wp^r),k,Q)} S_k(\Gamma_1^\Theta(\frn,\wp^r))(T_Q-\alpha),
\end{equation}
where each direct summand is stable under all Hecke operators and diamond operators. By \cite[Proposition 2.2]{Ha_DMBF}, we have
\[
\EV(\Gamma_1^\Theta(\frn,\wp^r),k,Q)\subseteq \bar{K}.
\]

Consider the natural inclusion
\[
\iota: S_k(\Gamma^p_1(\frn,\wp^r))\to S_k(\Gamma^p_1(\frn,\wp^{r+1})).
\]
For any $d\in (A/\frn\wp^{r+1}A)^\times$, we have 
\[
\langle d\rangle_{\frn\wp^{r+1}}\circ \iota=\iota\circ \langle d\bmod \frn\wp^r\rangle_{\frn\wp^{r}}.
\]
By the explicit description of $T_Q$ in \cite[\S3.1]{Ha_DMBF}, this implies that the map $\iota$ is compatible with all Hecke operators.

\begin{lem}\label{LemFiniteSlopeGenES}
	For any $\alpha\in \Cinf^\times$, the map $\iota$ induces an isomorphism
	\[
	S_k(\Gamma^p_1(\frn,\wp^r))(U_\wp-\alpha)\to S_k(\Gamma^p_1(\frn,\wp^{r+1}))(U_\wp-\alpha)
	\]
	which is compatible with all Hecke operators.
\end{lem}
\begin{proof}
	Write $\Gamma_r=\Gamma^p_1(\frn,\wp^r)$.
	Let $R(\wp)\subseteq A$ be a complete set of representatives of $A/\wp A$. We can write
	\[
	\Gamma_{r+1}\begin{pmatrix}
		1&0\\0&\wp
	\end{pmatrix}\Gamma_r=\coprod_{\beta\in R(\wp)} \Gamma_{r+1}\xi_\beta,\quad \xi_\beta=\begin{pmatrix}
	1&\beta\\0&\wp
	\end{pmatrix}.
	\] 
	Then we have a $\Cinf$-linear map
	\[
	s:S_k(\Gamma_{r+1})\to S_k(\Gamma_r),\quad f\mapsto \sum_{\beta\in R(\wp)} f|_k\xi_\beta
	\]
	satisfying the commutative diagram
	\[
	\xymatrix{
		S_k(\Gamma_r) \ar[r]^\iota\ar[d]_{U_\wp}& S_k(\Gamma_{r+1})\ar[d]^{U_\wp}\ar[ld]_s\\
			S_k(\Gamma_r) \ar[r]_\iota& S_k(\Gamma_{r+1}).
	}
	\]
	This implies $U_\wp=0$ on $\Coker(\iota)$. Since $\iota$ commutes with $U_\wp$, it induces an injection
	\[
	\iota_\alpha: S_k(\Gamma_r)(U_\wp-\alpha)\to S_k(\Gamma_{r+1})(U_\wp-\alpha).
	\] 
	By (\ref{EqnGenESDecomp}), we see that $\Coker(\iota_\alpha)$ is a $\Cinf$-subspace of $\Coker(\iota)$.
	Since $\alpha\neq 0$, 
	we have $\Coker(\iota_\alpha)=0$ and the map $\iota_\alpha$ is also surjective.
	Since $\iota$ is compatible with all Hecke operators, so is $\iota_\alpha$.
\end{proof}



\section{Freeness under the $\Theta_r$-action}\label{SecFree}

In this section, write $\Gamma=\Gamma_1(\frn\wp^r)$ and $\tilde{\Gamma}=\Gamma_1^p(\frn,\wp^r)$,
so that $\tilde{\Gamma}\rhd\Gamma$ and $\Theta_r=\tilde{\Gamma}/\Gamma$. 

Consider the right $\bZ[\tilde{\Gamma}]$-modules
\[
\tilde{L}_1:=\bZ[\bar{\cT}_1^{o,\tilde{\Gamma}\text{-}\st}]\quad \text{and} \quad
\tilde{L}_0:=\bZ[\cT_0^{\tilde{\Gamma}\text{-}\st}],
 \] 
on which the group $\tilde{\Gamma}$ acts by $s|_{\tilde{\gamma}}=\tilde{\gamma}^{-1}\circ s$ as before. 
Let $\tilde{\Lambda}_{1}\subseteq \cT_1^{o,\tilde{\Gamma}\text{-}\st}$ and $\tilde{\Lambda}_{0}\subseteq \cT_0^{\tilde{\Gamma}\text{-}\st}$ be complete sets of representatives
of 
\[
\tilde{\Gamma}\backslash \cT_1^{o,\tilde{\Gamma}\text{-}\st}/\{\pm 1\}\quad\text{and}\quad \tilde{\Gamma}\backslash \cT_0^{\tilde{\Gamma}
\text{-}\st},
\]
respectively. By \cite[Ch.~II, \S2.9, Theorem 13' (a)]{Se}, the sets $\tilde{\Lambda}_{1}$ and $\tilde{\Lambda}_{0}$ are finite. 
Moreover, we have
\begin{equation}\label{EqnResolAreFree}
\tilde{L}_1=\bigoplus_{e\in \tilde{\Lambda}_{1}}[e] \bZ[\tilde{\Gamma}],\quad 
\tilde{L}_0=\bigoplus_{v\in \tilde{\Lambda}_{0}}[v] \bZ[\tilde{\Gamma}].
\end{equation}
Since the right $\bZ[\Gamma]$-module $\bZ[\tilde{\Gamma}]$ is free of finite rank, we see that 
the sequence (\ref{EqnExactSt}) for $\tilde{\Gamma}$ is a split exact sequence of right $\bZ[\Gamma]$-modules. Hence we obtain an exact sequence
of $\Cinf$-vector spaces
\begin{equation}\label{EqnExactFreeResol}
	\xymatrix{
		0\ar[r] & \cV_k(\Gamma)\ar[r] & \tilde{L}_1\otimes_{\bZ[\Gamma]}V_k(\Cinf) \ar[r] & 
		\tilde{L}_0\otimes_{\bZ[\Gamma]}V_k(\Cinf) \ar[r] & 0,
	}
\end{equation}
which is compatible with natural right $\Theta_r$-actions.

\begin{lem}\label{LemThetaFreeOfResolution}
	For any $i\in\{0,1\}$, the right $\Cinf[\Theta_r]$-module $\tilde{L}_i\otimes_{\bZ[\Gamma]}V_k(\Cinf)$ is free of finite rank.
\end{lem}
\begin{proof}
	By (\ref{EqnResolAreFree}), we have an isomorphism of right $\Cinf[\Theta_r]$-modules
	\[
	\tilde{L}_i\otimes_{\bZ[\Gamma]}V_k(\Cinf)\to \bigoplus_{s\in \tilde{\Lambda}_{i}} \bZ[\tilde{\Gamma}]\otimes_{\bZ[\Gamma]}V_k(\Cinf).
	\]
	Thus we reduce ourselves to showing that the right $\Cinf[\Theta_r]$-module
	\[
	 \bZ[\tilde{\Gamma}]\otimes_{\bZ[\Gamma]}V_k(\Cinf)=\Ind_{\Gamma}^{\tilde{\Gamma}}\Res_{\Gamma}^{\tilde{\Gamma}}V_k(\Cinf)
	\]
	is free of finite rank.
	
	Let $\mathbf{1}$ be the trivial representation of $\Gamma$ over $\Cinf$. By the projection formula, we have a natural isomorphism of left $\Cinf[\tilde{\Gamma}]$-modules
	\[
	\pi:\Ind_{\Gamma}^{\tilde{\Gamma}}\Res_{\Gamma}^{\tilde{\Gamma}}V_k(\Cinf)\to (\Ind_{\Gamma}^{\tilde{\Gamma}}\mathbf{1})\otimes_{\Cinf} V_k(\Cinf)
	\]
	which is defined by
	\[
	\begin{aligned}
	 \bZ[\tilde{\Gamma}]\otimes_{\bZ[\Gamma]}V_k(\Cinf)&\to  (\bZ[\tilde{\Gamma}]\otimes_{\bZ[\Gamma]}\Cinf)\otimes_{\Cinf} V_k(\Cinf),\\
	 [\tilde{\gamma}]\otimes v&\mapsto [\tilde{\gamma}]\otimes 1 \otimes \tilde{\gamma}\circ v.
	 \end{aligned}
	\]
	On the target of the map $\pi$, we consider the trivial $\Theta_r$-action on $V_k(\Cinf)$ and
	the natural right $\Theta_r$-action on $\bZ[\tilde{\Gamma}]\otimes_{\bZ[\Gamma]}\Cinf$ given by
	\[
	([\tilde{\gamma}]\otimes 1)d=[\tilde{\gamma}\eta_{[1,d]}]\otimes 1
	\]
	for any $d\in \Theta_r$.
	Then the map $\pi$ is $\Theta_r$-equivariant. Moreover, with the latter $\Theta_r$-action, 
	we have an isomorphism of right $\Cinf[\Theta_r]$-modules
	\[
	\bZ[\tilde{\Gamma}]\otimes_{\bZ[\Gamma]}\Cinf\to \Cinf[\Theta_r],\quad [\eta_{[1,d]}]\otimes 1\mapsto [d].
	\]
	Hence we obtain an isomorphism of right $\Cinf[\Theta_r]$-modules
	\[
	\bZ[\tilde{\Gamma}]\otimes_{\bZ[\Gamma]}V_k(\Cinf)\simeq \Cinf[\Theta_r]^{\oplus k-1}.
	\]
	This concludes the proof.
\end{proof}

The following proposition is a generalization of \cite[Proposition 4.8]{Ha_DMO} which treats the case of $\frn=1$, $\wp=t$ and $k=2$.

\begin{prop}\label{PropFree}
	The $\Cinf[\Theta_r]$-module $S_k(\Gamma_1(\frn\wp^r))$ is free of finite rank.
\end{prop}
\begin{proof}
By Lemma \ref{LemQuotAction}, it is enough to show that the right $\Cinf[\Theta_r]$-module $\cV_k(\Gamma_1(\frn\wp^r))$ is free of finite rank.
By (\ref{EqnExactFreeResol}) and Lemma \ref{LemThetaFreeOfResolution}, 
we see that the right $\Cinf[\Theta_r]$-module $\cV_k(\Gamma_1(\frn\wp^r))$ is a direct summand of a free $\Cinf[\Theta_r]$-module of finite rank. 
Hence the proposition
follows from Lemma \ref{LemGrpRingLocal}.
\end{proof}



\section{Constancy of Hecke eigensystems}\label{SecConst}

For any $Q\in \Pi$ and $\alpha\in \bar{K}$, we denote by $m(\Gamma_1^\Theta(\frn,\wp^r),k,Q,\alpha)$ the multiplicity of $\alpha$ as an eigenvalue of $T_Q$ acting on
$S_k(\Gamma_1^\Theta(\frn,\wp^r))$. By definition, we have
\[
m(\Gamma_1^\Theta(\frn,\wp^r),k,Q,\alpha)=\dim_{\Cinf}S_k(\Gamma_1^\Theta(\frn,\wp^r))(T_Q-\alpha).
\]
For any $a\in \bQ\cup\{+\infty\}$, we denote by $d(\Gamma_1^\Theta(\frn,\wp^r),k,a)$ the multiplicity of eigenvalues of slope $a$ of $U_\wp$ acting on
$S_k(\Gamma_1^\Theta(\frn,\wp^r))$, so that
\begin{equation}\label{EqnSlopeMulti}
d(\Gamma_1^\Theta(\frn,\wp^r),k,a)=\sum_{\alpha\in \EV(\Gamma_1^\Theta(\frn,\wp^r),k,\wp),\ v_\wp(\iota_\wp(\alpha))=a} m(\Gamma_1^\Theta(\frn,\wp^r),k,\wp,
\alpha).
\end{equation}

\begin{thm}\label{ThmMultiTheta}
	For any $Q\in \Pi$ and $\alpha\in \bar{K}$, we have
	\[
	m(\Gamma_1(\frn\wp^r),k,Q,\alpha)=|\Theta| m(\Gamma_1^\Theta(\frn,\wp^r),k,Q,\alpha).
	\]
	In particular, for any $a\in \bQ\cup\{+\infty\}$ we have
	\[
	\begin{aligned}
	\EV(\Gamma_1(\frn\wp^r),k,Q)&=\EV(\Gamma_1^\Theta(\frn,\wp^r),k,Q),\\
	 d(\Gamma_1(\frn\wp^r),k,a)&=|\Theta|d(\Gamma_1^\Theta(\frn,\wp^r),k,a).
	\end{aligned}
	\]
\end{thm}
\begin{proof}
	Since (\ref{EqnGenESDecomp}) is a decomposition as a $\Cinf[\Theta_r]$-module, Proposition \ref{PropFree} and Lemma \ref{LemGrpRingLocal} imply that for any $
	\alpha\in \bar{K}$, the $\Cinf[\Theta_r]$-module $S_k(\Gamma_1(\frn\wp^r))(T_Q-\alpha)$
	is free of finite rank. Write
	\[
	S_k(\Gamma_1(\frn\wp^r))(T_Q-\alpha)=\Cinf[\Theta_r]^{\oplus m}
	\]
	with some integer $m\geq 0$.
	By (\ref{EqnThetaFixed}), this implies
	\[
	S_k(\Gamma_1^\Theta(\frn,\wp^r))(T_Q-\alpha)=(\Cinf[\Theta_r]^\Theta)^{\oplus m}.
	\]
	Since $\dim_{\Cinf}\Cinf[\Theta_r]^\Theta=[\Theta_r:\Theta]$, we obtain the first equality. 
	
	Since $\alpha\in \EV(\Gamma_1(\frn\wp^r),k,Q)$ if and only if $m(\Gamma_1(\frn\wp^r),k,Q,\alpha)\neq 0$ and similarly for
	$\Gamma_1^\Theta(\frn,\wp^r)$, the second equality follows. Then (\ref{EqnSlopeMulti}) implies the third one.
\end{proof}

Note that $\Gamma_1^p(\frn,\wp)=\Gamma_1(\frn\wp)$. 

\begin{cor}\label{CorMultiSlope}
	For any $\alpha\in \bar{K}^\times$ and $a\in \bQ$, we have
	\[
	\begin{aligned}
		m(\Gamma_1(\frn\wp^r),k,\wp,\alpha)&=q^{(r-1)\deg(\wp)}m(\Gamma_1(\frn\wp),k,\wp,\alpha),\\
		d(\Gamma_1(\frn\wp^r),k,a)&=q^{(r-1)\deg(\wp)}d(\Gamma_1(\frn\wp),k,a).
	\end{aligned}
	\]
\end{cor}
\begin{proof}
	For any $\alpha\in \bar{K}^\times$, Lemma \ref{LemFiniteSlopeGenES} yields 
	\[
	\begin{aligned}
	m(\Gamma_1^p(\frn,\wp^r),k,\wp,\alpha)&=m(\Gamma_1(\frn\wp),k,\wp,\alpha),\\
	\EV(\Gamma_1^p(\frn,\wp^r),k,\wp)\setminus\{0\}&=\EV(\Gamma_1(\frn\wp),k,\wp)\setminus\{0\}.
	\end{aligned}
	\]
	Since $|\Theta_r|=q^{(r-1)\deg(\wp)}$, Theorem \ref{ThmMultiTheta} and (\ref{EqnSlopeMulti}) yield the corollary.
\end{proof}

Let $\lambda=(\lambda_Q)_{Q\in \Pi}\in \bar{K}^\Pi$. We say that $\lambda$ is a Hecke eigensystem appearing in $S_k(\Gamma_1^\Theta(\frn,\wp^r))$
if there exists a nonzero element $f\in S_k(\Gamma_1^\Theta(\frn,\wp^r))$ satisfying $f|T_Q=\lambda_Q f$ for all $Q\in \Pi$. Let
\[
\begin{aligned}
	S_k(\Gamma_1^\Theta(\frn,\wp^r))(\lambda)&=\bigcap_{Q\in\Pi}S_k(\Gamma_1^\Theta(\frn,\wp^r))(T_Q-\lambda_Q),\\
	m(\Gamma_1^\Theta(\frn,\wp^r),k,\lambda)&=\dim_{\Cinf}S_k(\Gamma_1^\Theta(\frn,\wp^r))(\lambda).
\end{aligned}
\]
Since $\Pi$ is countably infinite, we can choose a bijection
\[
\bZ_{>0}\to \Pi,\quad i\mapsto Q_i.
\]

\begin{lem}\label{LemGenESEigenform}
	We have $S_k(\Gamma_1^\Theta(\frn,\wp^r))(\lambda)\neq 0$ if and only if $\lambda$ is a Hecke eigensystem appearing in $S_k(\Gamma_1^\Theta(\frn,\wp^r))$.
\end{lem}
\begin{proof}
	Write $V_0=S_k(\Gamma_1^\Theta(\frn,\wp^r))(\lambda)$. Since any nonzero eigenform in $S_k(\Gamma_1^\Theta(\frn,\wp^r))$ with Hecke eigensystem $\lambda$
	lies in $V_0$, the ``if'' part is clear.
	
	Suppose $V_0\neq 0$. Since $T_{Q_1}-\lambda_{Q_1}$ is nilpotent on $V_0$, we have 
	\[
	V_1:=\Ker(T_{Q_1}-\lambda_{Q_1}: V_0\to V_0)\neq 0,
	\]
	on which $T_{Q_2}-\lambda_{Q_2}$ is nilpotent.
	
	Thus we can inductively construct a decreasing sequence of nonzero $\Cinf$-vector spaces
	\[
	V_0\supseteq V_1\supseteq \cdots \supseteq V_i\supseteq V_{i+1}\supseteq \cdots
	\]
	such that for any $i\in \bZ_{>0}$, we have 
	\[
	V_{i}=\Ker(T_{Q_{i}}-\lambda_{Q_{i}}:V_{i-1}\to V_{i-1}).
	\]
	Since $\dim_{\Cinf}V_0$ is finite, the sequence is stationary and there exists a positive integer $j$ such that
	\[
	V_j=\{f\in V_0\mid f|(T_Q-\lambda_Q)=0\text{ for all }Q\in\Pi\}.
	\]
	Since $V_j\neq 0$, it follows that $\lambda$ is a Hecke eigensystem appearing in $S_k(\Gamma_1^\Theta(\frn,\wp^r))$.
\end{proof}

\begin{thm}\label{ThmConstancy}
	Let $\lambda=(\lambda_Q)_{Q\in\Pi}\in \bar{K}^\Pi$ satisfy $\lambda_\wp\neq 0$. Then we have
	\[
	m(\Gamma_1(\frn\wp^r),k,\lambda)=q^{(r-1)\deg(\wp)}m(\Gamma_1(\frn\wp),k,\lambda).
	\]
	In particular, $\lambda$ is a Hecke eigensystem appearing in $S_k(\Gamma_1(\frn\wp^r))$ 
	if and only if it is a Hecke eigensystem appearing in $S_k(\Gamma_1(\frn\wp))$. 
\end{thm}
\begin{proof}
	Write $S_0=S_k(\Gamma_1(\frn\wp^r))$. 
	For any $i\in \bZ_{>0}$, we inductively define
	\[
	S_i:=S_{i-1}(T_{Q_i}-\lambda_{Q_i}),
	\]
	so that we have a decreasing sequence of $\Cinf$-vector spaces
	\[
	S_0\supseteq S_1\supseteq \cdots \supseteq S_i\supseteq S_{i+1}\supseteq \cdots.
	\]
	Since $\dim_{\Cinf}S_0$ is finite, the sequence is stationary and there exists a positive integer $j$ satisfying
	\[
	S_j=S_0(\lambda)=\bigcap_{Q\in \Pi} S_k(\Gamma_1(\frn\wp^r))(T_{Q}-\lambda_Q).
	\]
	Then we have
	\[
	S_j^{\Theta_r}=S_k(\Gamma_1^p(\frn,\wp^r))(\lambda)=\bigcap_{Q\in \Pi} S_k(\Gamma_1^p(\frn,\wp^r))(T_{Q}-\lambda_Q).
	\]
	
	For any $i\in \bZ_{>0}$, the $\Cinf$-vector space $S_{i-1}$ is the direct sum of the generalized eigenspaces of $T_{Q_{i}}$ acting on it.
	Since diamond operators commute with all Hecke operators, 
	these direct summands are stable under the $\Theta_r$-action. In particular,
	we see that $S_i$ is a direct summand of $S_{i-1}$ 
	as a $\Cinf[\Theta_r]$-module.
	
	By Proposition \ref{PropFree} and Lemma \ref{LemGrpRingLocal}, this shows that
	the $\Cinf[\Theta_r]$-module $S_j$ is free of finite rank. 
	Since $\dim_{\Cinf}\Cinf[\Theta_r]=q^{(r-1)\deg(\wp)}$ and $\dim_{\Cinf}\Cinf[\Theta_r]^{\Theta_r}=1$, we have
	\[
	m(\Gamma_1(\frn\wp^r),k,\lambda)=q^{(r-1)\deg(\wp)}m(\Gamma_1^p(\frn,\wp^r),k,\lambda).
	\]
	Since $\lambda_\wp\neq 0$, Lemma \ref{LemFiniteSlopeGenES} implies 
	\[
	m(\Gamma_1^p(\frn,\wp^r),k,\lambda)=m(\Gamma_1(\frn\wp),k,\lambda).
	\]
	Hence we obtain the former assertion of the theorem. The latter assertion follows from the former and Lemma \ref{LemGenESEigenform}.
\end{proof}

\begin{rmk}\label{RmkSkipFreeness}
	Note that the only irreducible representation of $\Theta_r$ over $
	\Cinf$ is the trivial representation. Hence, if $S_k(\Gamma_1(\frn\wp^r))(\lambda)$ is nonzero, then its $\Theta_r$-fixed part
	$S_k(\Gamma_1^p(\frn,\wp^r))(\lambda)$ is also nonzero. This also leads to a proof of the latter assertion of Theorem \ref{ThmConstancy}.  
\end{rmk}

\begin{rmk}\label{RmkHalo}
	Corollary \ref{CorMultiSlope} shows a pattern of slopes different from that of elliptic modular forms.
	Let $k\geq 2$ and $r\geq 2$ be integers. Let $\chi:(\bZ/2^r\bZ)^\times\to \bC^\times$ be a Dirichlet character of conductor $2^r$
	satisfying $\chi(-1)=(-1)^k$. Then Buzzard--Kilford \cite[Corollary of Theorem B, (i)]{BK} proved that 
	the $2$-adic slopes appearing in $S_k(\Gamma_0(2^r),\chi)$ form
	an arithmetic progression such that for any integer $i\geq 1$, the $i$th term goes to zero when $r$ goes to infinity. 
	By the assumption on the conductor 
	of $\chi$, this is a behavior of slopes at points on the $2$-adic weight space near the boundary.
	In contrast, by Corollary \ref{CorMultiSlope}, the set of finite slopes in Drinfeld cuspforms
	of level $\Gamma_1(\wp^r)$ is constant when $r$ varies.
	
	A possible explanation for the 
	difference
	is that, since there is no nontrivial
	character $
	\Theta_r\to \Cinf^\times $, every character $(A/\wp^r A)^\times\to \Cinf^\times$ factors through $(A/\wp A)^\times$ and such a boundary behavior does not occur 
	for Drinfeld modular forms.
	
\end{rmk}



\end{document}